\documentclass[12pt]{amsart}
\usepackage{amsfonts}
\usepackage{latexsym,amsmath,amsthm,amssymb,amsfonts,mathrsfs}
\usepackage{graphicx}
\numberwithin{equation}{section}

\newtheorem{thm}{Theorem}[section]

\newtheorem{lem}{Lemma}[section]
\newtheorem{prop}{Proposition}[section]

\theoremstyle{definition}

\theoremstyle{remark}

\newcommand{\il}{\langle}
\newcommand{\li}{\rangle}

\begin{document}

\title[The space of compact self-shrinking solutions]{The space of compact self-shrinking solutions to the Lagrangian Mean Curvature Flow in $\mathbb C^2$}
\author{Jingyi Chen and John Man Shun Ma}
\thanks {The first author is partially supported by an NSERC grant.}

\address{Department of Mathematics,
The University of British Columbia,
Canada} 
\email{jychen@math.ubc.edu.ca}
\email{johnma@math.ubc.edu.ca}

\begin{abstract}
Let $F_n :(\Sigma, h_n) \to \mathbb C^2$ be a sequence of conformally immersed Lagrangian self-shrinkers with a uniform area upper bound to the mean curvature flow, and suppose that the sequence of metrics $\{h_n\}$ converges smoothly to a Riemannian metric $h$. We show that a subsequence of $\{F_n\}$ converges smoothly to a branched conformally immersed Lagrangian self-shrinker $F_\infty : (\Sigma, h)\to \mathbb C^2$. When the area bound is less than $16\pi$, the limit $F_\infty$ is an embedded torus. When the genus of $\Sigma$ is one, we can drop the assumption on convergence $h_n\to h$. When the genus of $\Sigma$ is zero,  we show that there is no branched immersion of $\Sigma$ as a Lagrangian shrinker, generalizing the rigidity result of \cite{Sm0} in dimension two by allowing branch points. 
\end{abstract}

\date{June 24, 2014}

\maketitle

\section{Introduction}
Let $\Sigma$ be an orientable smooth $m$ dimensional manifold without boundary and let $F$ be an immersion of $\Sigma$ into $\mathbb R^{m+k}$. The smooth mapping $F$ is called a {\sl self-shrinker} if it satisfies the equation
\begin{equation} \label{self similar eqn}
\vec H = -F^\perp,
\end{equation}
where $\vec H$ is the mean curvature vector and $F^\perp$ is the normal component of the position vector $F$. A self-shrinker gives rise to a solution, via the scaling  $F_t = \sqrt{1-2 t}\, F$ for $0 \leq t <1/2$, of the mean curvature flow (MCF) in ${\mathbb R}^{m+k}$:
\[\bigg(\frac{\partial F_t}{\partial t}\bigg)^\perp = \vec H(F_t).\]
MCF typically develops singularities in finite time. By parabolic rescaling at a singularity, Huisken's monotonicity formula \cite{H} shows that a subsequence of the mean curvature flow converges to a self-shrinker when the singularity is of type I. Ultimately, the classification of self-shrinkers would be important in understanding the MCF. 

When $m+k=2m$, we can identify $\mathbb R^{2m}$ with $\mathbb C^m$ and consider the Lagrangian immersions. It is shown in \cite{Sm1} that the Lagrangian condition is preserved along the MCF. Therefore, the shrinkers arise from Lagrangian MCF are Lagrangian submanifolds. When $m=1$, the Lagrangian condition is automatically satisfied by smooth curves and self shrinking solutions are studied in \cite{AL}. 

In this paper, we will restrict our attention to compact Lagrangian self shrinkers in ${\mathbb C}^2$. For arbitrary $m$, Smoczyk \cite{Sm0} showed that there is no Lagrangian self-shrinking immersion in $\mathbb C^m$ with first Betti number equals zero. In particular,  there is no immersed Lagrangian self shrinking sphere in $\mathbb C^2$. To establish compactness properties of the moduli space of compact Lagrangian shrinkers, it is crucial to generalize Smoczyk's result to {\it branched} Lagrangian immersions as the limit of a sequence of immersions may not be an immersion anymore. 

Indeed, the rigidity holds for branched Lagrangian shrinking spheres in ${\mathbb C}^2$, we have: 

\begin{thm} \label{no S2 lag shrinker}
There does not exist any branched conformal Lagrangian self-shrinking sphere in $\mathbb C^2$. 
\end{thm}

At any immersed point,  the mean curvature form $\alpha_H = \iota_{\vec H} \omega$ of the Lagrangian shrinker satisfies a pair of differential equations (\ref{da=0}) and  (\ref{d*a eqn}) that form a first order elliptic system. The key ingredient in the proof of Theorem \ref{no S2 lag shrinker} is to show that $\alpha_H$ extends smoothly across the branch points. The self-shrinker equation (\ref{self similar eqn}) yields a $L^\infty$ bound on  $\alpha_H$ and this is useful to show that $\alpha_H$ satisfies the first order elliptic system as a weak solution distributionally. Smoothness of the extended mean curvature form then follows from the elliptic theory. 
  
The main application of Theorem \ref{no S2 lag shrinker} is to derive compactness results of compact Lagrangian self-shrinkers. 

Let $\Sigma$ be a fixed compact oriented smooth surface and $F_n:\Sigma \to \mathbb C^2$ be a sequence of Lagrangian immersed self-shrinkers in $\mathbb C^2$.  Let $h_n$ be the Riemannian metric on $\Sigma$ which is conformal to the pull back metric $F_n^*\il \cdot, \cdot\li$  on $\Sigma$ such that either it has constant Gauss curvature $-1$ if the genus of $\Sigma$ is greater than one, or $(\Sigma , h_n)$ is $ \mathbb C / \{1, a+ bi\}$ with the flat metric, where $-\frac{1}{2} < a\leq \frac{1}{2}$, $b\geq 0$, $a^2 + b^2 \geq 1$ and $a\geq 0$ whenever $a^2 + b^2  =1$. It is well known that the Moduli space of the conformal structures on $\Sigma$ are parametrized by metrics of the above form. 

We now state our compactness result. 

\begin{thm} \label{main theorem}
Let $F_n :(\Sigma, h_n) \to \mathbb C^2$ be a sequence of conformally immersed Lagrangian self-shrinkers with a uniform area upper bound $\Lambda$. Suppose that the sequence of metrics $\{h_n\}$ converges smoothly to a Riemannian metric $h$ on $\Sigma$. Then there is a subsequence of $\{F_n\}$ that converges smoothly to a branched conformally immersed Lagrangian self-shrinkers $F_\infty :(\Sigma, h) \to \mathbb C^2$.
\end{thm}

Note that  there is a universal positive lower bound on the extrinsic diameters for two dimensional compact shrinkers (cf. section 4). The limit $F_\infty$ cannot be a constant map.

The proof of Theorem \ref{main theorem} uses the observation that self-shrinkers are in fact minimal immersions into $(\mathbb C^2, G)$, where $G$ is a metric on $\mathbb C^2$ conformal to the Euclidean metric. 
The advantage of this viewpoint is that we are then able to use the bubble tree convergence of harmonic maps developed in \cite{P}. In particular, Theorem \ref{no S2 lag shrinker} shows that no bubble is formed during the process and thus the convergence is smooth.


It is interesting to compare Theorem \ref{main theorem} with the compactness results of Colding-Minicozzi \cite{CM1} on self-shrinkers in $\mathbb R^3$ and of Choi-Schoen \cite{CS}, Fraser-Li \cite{FL} on embedded minimal surfaces in three dimensional manifold $N$ (with or without boundary) with nonnegative Ricci curvature. In their cases, they use the local singular compactness theorem (see for example Proposition 2.1 of \cite{CM1}), which says that for any sequence of embedded minimal surfaces $\Sigma_n$ in $N$ with a uniform (global or local) area or genus upper bound, there is an embedded minimal surface $\Sigma$ such that a subsequence of $\Sigma_n$ converges smoothly and locally (with finite multiplicities) to $\Sigma$ away from finitely many points in $\Sigma$. A removable singularity theorem, which is based on the maximum principle 
and is true only in the codimension one case, is needed to prove the local singular compactness theorem. Hence similar statement is not available in higher codimensions. For embedded minimal surfaces in a 3-manifold, Colding and Minicozzi have proven deep compactness results (\cite{CM2} and the reference therein).

In the case of arbitrary codimension, the regularity of the limit of a sequence of minimal surfaces (as a stationary varifold  given by Geometric Measure Theory) is a subtle issue. Not much is known except for some special cases, such as Gromov's compactness theorem on pseudo holomorphic curves \cite{G}.

In our situation, we do not assume embeddedness. In fact, if we assume that  $F_n$ is a Lagrangian embedding, then by a result of Whitney (\cite{W}, see also \cite{ALP}), $\Sigma$ must be a torus as the self-intersection number $\Sigma\cdot\Sigma$ is twice of the genus of $\Sigma$ minus 2. When we treat the immersions $F_n$ as conformal harmonic mappings, even impose the assumption that the conformal structures induced from the immersions  stay in a compact domain of the Teichm\"uller space,  the sequence $F_n$ would subsequentially converge only to a {\it branched} immersion on each component in the bubble tree convergence. True branch points might exist for the limiting map $F_\infty$. On the other hand, if $\Lambda<16\pi$ in Theorem \ref{main theorem} then the Willmore energy of each immersion $F_n$ is less than $8\pi$ for shrinkers (cf. \cite{CL}), hence by an inequality of Li-Yau (Theorem 6 in \cite{LY}), $F_n$ must be an embedding and in turn $\Sigma$ is a torus since it is Lagrangian.

When $\Sigma$ has genus one, we can drop the assumption on the convergence of conformal structures in Theorem \ref{main theorem}. Moreover, if the area bound is strictly less than $16\pi$, the limiting surface has no branched points and we have smooth compactness.

\begin{thm} \label{main thm 2}
Let $F_n : T \to \mathbb C^2$ be a sequence of Lagrangian self-shrinking tori with a uniform area upper bound $\Lambda$. Then (i) there is a subsequence of $\{F_n\}$ that converges smoothly to a branched Lagrangian self-shrinking torus $F_\infty:T\to{\mathbb C}^2$, and (ii) when $\Lambda<16\pi$, $F_\infty$ is an embedding. 
\end{thm}

Our strategy of proving Theorem \ref{main thm 2} is to rule out the possibility of degeneration of the conformal structures induced by $F_n$ on $T$. To do this we need the general bubble tree convergence results in \cite{CLi}, \cite{CT} that allow the conformal structures degenerate. The key observation is that when the conformal structures degenerate, some homotopically nontrivial closed curves in the torus are pinched to points. Thus the limiting surface is a finite union of spheres (arising from collapsion of the closed curves and from the bubbles at the energy concentration points) which are branched Lagrangian shrinkers, but Theorem \ref{no S2 lag shrinker} forbids this. The desired result then follows from Theorem \ref{main theorem}.

\section{Preliminaries}
Given an immersion $F : \Sigma \to \mathbb R^{m+k}$, let $g$ be the induced metric $F^*\il \cdot, \cdot\li$ on $\Sigma$. In local coordinates $(x^1. \cdots x^m)$ of $\Sigma$, the second fundamental form $A$ and the mean curvature vector $\vec H$ are given by 

$$ 
A_{ij}= \big(\partial_{ij} F \big)^\perp,\ \ \vec H = \sum_{i,j=1}^m g^{ij}A_{ij} = \Delta_g F 
$$
where $\Delta_g$ is the Laplace-Beltrami operator in the induced metric $g$. 

Let $G$ be a metric on $\mathbb R^{m+k}$ defined by 
\begin{equation} \label{dfn of G}
G(\cdot, \cdot) = e^{-\frac{|x|^2}{m}} \il \cdot , \cdot\li  .
\end{equation}

The following lemma is well-known and is proved by Angenent \cite{A} in the case of hypersurface. The proof can be generalized to arbitrary codimension.

\begin{lem} \label{self similar iff minimal wrt g}
The immersion $F$ satisfies equation (\ref{self similar eqn}) if and only if $F$ is a minimal immersion with respect to the metric $G$ (\ref{dfn of G}) on $\mathbb R^{m+k}$. 
\end{lem}

Next we consider immersions $F:\Sigma \to \mathbb C^m$. We identify $\mathbb R^{2m}$ with $\mathbb  C^m$ by $z^i = u^i + \sqrt{-1} v^i, i=1,\cdots,m$. The map $F$ is called a Lagrangian immersion if $F^*\omega=0$, where 
$$
\omega = \sum_{i=1}^m du^i \wedge dv^i
$$
is the standard symplectic form on $\mathbb C^m$. Let $J$ be the standard complex structure on $\mathbb C^m$. It is known that $\omega$, $J$ and $\il \cdot, \cdot \li$ on $\mathbb C^m$ are related by 
\begin{equation} \label{compatibility of g, w, J}
\il JX, JY\li = \il X, Y\li,\ \ \omega(X, Y) = \il JX, Y\li 
\end{equation}
for any $X,Y\in T_x{\mathbb C}^m, x\in {\mathbb C}^m$. Thus $F$ is Lagrangian if and only if $J$ sends the tangent vectors of $\Sigma$ to the normal vectors in ${\mathbb C}^m$. 

In particular, by equation (\ref{compatibility of g, w, J}), $J\vec H(x)$ is tangent to $F(\Sigma)$ at $x$ for any Lagrangian immersion $F$. The mean curvature form $\alpha_H$ is the 1-form on $\Sigma$ defined by:  for all $x\in \Sigma$ and $Y\in T_x \Sigma$, 
\begin{equation} \label{a dfn}
\alpha_H (Y)= \omega (\vec H (x), (F_*)_xY)= g (J\vec H (x), (F_*)_xY ).
\end{equation}
Let $d$ denotes the exterior differentiation of $\Sigma$. It is shown in \cite{D} that $\alpha_H$ is closed:
\begin{equation} \label{da=0}
d\alpha_H = 0\ .
\end{equation}

If $F$ is also a self-shrinker, we have the following equation for $\alpha_H$ (\cite{CL}, p.1521). For the convenience of the reader, we include a proof. 

\begin{lem}  \label{d^*a lemma} Let $F$ be a Lagrangian self-shrinker. Then $\alpha_H$ satisfies
\begin{equation} \label{d*a eqn}
d^* \alpha_H =  -\frac 12 \alpha_H (\nabla_g |F|^2) ,
\end{equation}
where $d^*$ is the formal adjoint of $d$ on $\Sigma$ with respect to $g$. 
\end{lem}

\begin{proof}
In local coordinates, 
$$
\alpha_H = \sum_{i=1}^m(\alpha_H)_i dx^i
$$ 
where the coefficients are given by 
\begin{eqnarray*}
(\alpha_H)_i &=&  \il J\vec H, \partial_i F\li \\
&=& -\il JF^\perp , \partial _i F\li  \\
&=& -\il JF, \partial _i F\li.
\end{eqnarray*}
Note that we have used the self-shrinker equation (\ref{self similar eqn}) in the second equality. Now fix a point $p\in \Sigma$ and take the normal coordinates at $p$. At $p$, $g_{ij} = g^{ij} = \delta_{ij}$ and $\Gamma_{ij}^k = 0$. We now calculate at $p$,  
\[\begin{split}
d^*\alpha_H &= -\sum_{i,j=1}^m g^{ij} (\alpha_H)_{i;j} \\
&= - \sum_{i=1}^m (\alpha_H)_{i,i} \\
&=  \sum_{i=1}^m \partial_i \il JF, \partial _i F \li \\
&=  \sum_{i=1}^m \left(\il J \partial_i F, \partial _i F \li  + \il JF , \partial^2_{ii} F \li  \right)
\end{split}\]
Using $J \partial_i F \perp \partial _i F$ and $\partial^2_{ii} F = A_{ii}$ for each $i$ at $p$, 
\[
\begin{split}
d^*\alpha_H &=  \sum_{i=1}^m \il JF, A_{ii} \li \\
&= \il JF,\vec H \li \\
&= -\il F, J\vec H\li \\
&= - \alpha_H (F^\top) \\
&= -\frac 12 \alpha_H(\nabla_g |F|^2),
\end{split}
\]
where $F^\top$ is the tangential component of $F$ and we have used 
$$
F^\top=\sum^m_{i=1}\langle F, \partial_i F\rangle \,\partial_i F = \frac{1}{2}\nabla_g|F|^2
$$ 
at $p$. 
\end{proof}

\section{ Lagrangian conformal branched immersion}
From now on we consider $m=2$. Let $(\Sigma, g_0)$ be a smooth Riemann surface. A smooth map $F :\Sigma \to \mathbb C^2$ is called a branched conformal immersion if 
\begin{enumerate}
\item there is a discrete set $B \subset \Sigma$ such that 
\[F: \Sigma \setminus B \to \mathbb C^2\]
 is an immersion,
\item there is a function $\lambda : \Sigma \to [0,\infty)$ such that $g:= F^* \il \ , \ \li = \lambda  g_0$ on $\Sigma$, where $\il \ , \ \li$ is the Euclidean metric on $\mathbb C^2$ and $\lambda$ is zero precisely at $B$, and
\item the second fundamental form $A$ on $\Sigma\setminus B$ satisfies $|A|_g \in L^2 (K\setminus B, d\mu)$ for any compact domain $K$ in $\Sigma$, where $|\cdot |_g$ and $d\mu$ are respectively the norm and the area element with respect to $g$ (Note $g$ defines a Riemannian metric on $\Sigma\setminus B$ but not on $\Sigma$).
\end{enumerate}

Elements in $B$ are called the branch points of $F$. A conformal branched immersion $F:\Sigma \to \mathbb C^2$ is called Lagrangian (resp. self-shrinking) if it is Lagrangian (resp. self-shrinking) when restricted to $\Sigma\setminus B$. Note that when $F$ is Lagrangian, the mean curvature form $\alpha_H$ is defined only on $\Sigma \setminus B$. 

The following proposition is the key result on removable singularity of $\alpha_H$. Note that in this proposition we do not assume $\Sigma$ to be closed. 

\begin{prop} \label{a extend across B}
Let $F:\Sigma \to \mathbb C^2$ be a branched conformal Lagrangian self-shrinker with the set of branch points $B$. Then there is a smooth one form $\tilde \alpha$ on $\Sigma$ which extends $\alpha_H$ and $d\tilde \alpha=0$ on $\Sigma$. 
\end{prop}

\begin{proof}
The result is local, so it suffices to consider $\Sigma$ to be the unit disc $\mathbb D$ with a branch point at the origin only. Let $(x, y)$ be the local coordinate of $\mathbb D$. We write $\alpha_H  = a dx + bdy$ for some smooth function $a$ and $b$ on the punctured disk $\mathbb D^* = \{ z\in \mathbb D: z\neq 0\}$. Let
\begin{equation} \label{dfn of delta and D}
\text{div}( \alpha_H) =   \frac{\partial a}{\partial x} + \frac{\partial b}{\partial y}   , \ \ \ \ D  = \bigg(\frac{\partial }{\partial x}, \frac{\partial }{\partial y} \bigg)
\end{equation}
be the divergence and the gradient with respect to $\delta_{ij}$. As $F$ is conformal, $g_{ij} = \lambda \delta_{ij}$, where $\lambda = \frac{1}{2}| DF|^2$ and $\delta_{ij}$ is the euclidean metric on $\mathbb D$. By restricting to a smaller disk if necessary, we assume that the image $|F|$ and $\lambda$ are bounded.  By equation (\ref{self similar eqn}) we have $$|\vec H| \leq |F|,$$ where $|\cdot|$ is taken with respect to $\il \cdot, \cdot\li$ on $\mathbb C^2$. As $g = F^*\il \cdot, \cdot\li$, using equation (\ref{a dfn}) we see that 
$$
|\alpha_H|_g \leq C\,\,\text{ on }\mathbb D^*.
$$
 Using $g_{ij} = \lambda \delta_{ij}$ and $n=2$, we have 
$$
\nabla_g = \lambda^{-1} D\, \, \text{and}\, \,d^*_g  = - \lambda^{-1} \text{div}.
$$ 
Hence equation (\ref{d*a eqn}) is equivalent to
\begin{equation} \label{d*a eqn euclidean}
\text{div}( \alpha_H) =  \frac 12 \alpha_H (D |F|^2 ).
\end{equation}
Moreover, as $|\alpha_H|_g = \sqrt{\lambda^{-1} (a^2 + b^2)}$, we also have 
$$|a|, |b| \leq  |\alpha_H|_g \sqrt\lambda\leq C\sqrt\lambda.$$
Thus $|a|, |b|$ are bounded on $\mathbb D^*$. To simplify notations, let
$$P = \frac 12 \alpha_H (D|F|^2).$$
Note that $P$ is also bounded on $\mathbb D ^*$. 

Both equations (\ref{da=0}) and (\ref{d*a eqn euclidean}) are satisfied pointwisely in $\mathbb D^*$. We now show that they are satisfied in the sense of distribution on $\mathbb D$. That is, for all test functions $\phi \in C^\infty_0(\mathbb D)$, 
\begin{equation} \label{da=0 dist}
\int_\mathbb{D} \alpha_H \wedge d\phi = 0
\end{equation}
and 
\begin{equation} \label{d* a eqn dist}
\int_\mathbb{D} \alpha_H (D\phi) dxdy =  - \int_\mathbb{D}P \phi dxdy .
\end{equation}
Note that all the integrands in equation (\ref{da=0 dist}) and (\ref{d* a eqn dist}) are integrable, since $a, b$ and $P$ are in $L^\infty (\mathbb D)$. 

First we show (\ref{da=0 dist}). Let $\psi_r \in C^\infty_0(\mathbb D)$, $r<1/2$, be a cutoff function such that $0\leq \psi \leq 1$,  $|D\psi| \leq 2/r$ and
\[\psi(x) = \begin{cases} 1 & \text{when }|x|\geq 2r,\\
 0 & \text{when } |x|\leq r.
\end{cases}\]
Then $\phi \psi_r \in C^\infty_0(\mathbb D^*)\cap C^\infty(\mathbb D)$. Using Stokes' theorem and equation (\ref{da=0}), we have
$$
 0 =\int_{\mathbb D ^*} d (\phi\psi_r \alpha_H) =  \int_{\mathbb D^*} d(\phi \psi_r)\wedge   \alpha_H  .
$$
This implies 
\begin{equation} \label{da =0 dist r}
 \int_{\mathbb D^* } \psi_r  \alpha_H \wedge d\phi= -\int_{\mathbb D^*} \phi  \alpha_H \wedge  d\psi_r .
\end{equation}
Since $\psi_r \to 1$ on $\mathbb D^* $ as $r\to 0$ and $\alpha_H$, $d\phi$ are bounded,
$$
\lim_{r\to 0} \int_{\mathbb D^*} \psi_r \alpha_H \wedge d\phi = \int_{\mathbb D^*} \alpha_H \wedge d\phi =\int_\mathbb{D} \alpha_H \wedge d\phi
$$
by Lebesque's dominate convergence theorem. To estimate the right hand side of equation (\ref{da =0 dist r}), note that as $d\psi_r$ has support on $\mathbb D_{2r} \setminus \mathbb D_r$, where $\mathbb D_s$ denotes the disk of radius $s$. Hence
\begin{equation} \label{estimate of da = 0}
\bigg| \int_\mathbb{D} \phi \alpha_H \wedge d\psi_r \bigg|\leq \frac{4 C \sup |\phi|}{r} \int_{\mathbb D_{2r}\setminus \mathbb D_r} dx  = 12\pi C \sup|\phi| r \to 0
\end{equation}
as $r \to 0$. Thus equation (\ref{da=0 dist}) holds.

To show equation (\ref{d* a eqn dist}), we use the same cutoff function $\psi_r$. Then $\phi \psi_r \in C^\infty_0(\mathbb D^*)\cap C^\infty(\mathbb D)$. By the divergence theorem, 
$$
0 = \int_{\mathbb D^*} \text{div}(\phi\psi_r \alpha_H) \,dxdy 
= \int_{\mathbb D^*} \alpha_H \big( D(\phi\psi_r ) \big) \, dxdy + \int_{\mathbb D^*} \phi\psi_r \text{div}(\alpha_H) \,dxdy .
$$
Now we use equation (\ref{d*a eqn euclidean}) to conclude 
\begin{equation} \label{d*a eqn dist proof}
-\int_{\mathbb D^*} P\phi\psi_r \,dxdy = \int_{\mathbb D^*} \psi_r \alpha_H (D\phi) \,dxdy + \int_{\mathbb D^*} \phi \alpha_H (D\psi_r)\,dxdy .
\end{equation}
Similarly, we can estimate the second term on the right hand side of equation (\ref{d*a eqn dist proof}) as for equation (\ref{estimate of da = 0}):
\begin{equation}\label{estimate of d^*a}
\bigg| \int_{\mathbb D^*} \phi \alpha_H (D\psi_r)\,dxdy\bigg| \leq  12\pi C \sup|\phi| r .
\end{equation}
Using Lebesque's dominate convergence theorem again, we can set $r\to 0$ in equation (\ref{d*a eqn dist proof}) to arrive at equation (\ref{d* a eqn dist}). 

Writing $\alpha_H = a dx + bdy$, equations (\ref{da=0 dist}) and (\ref{d* a eqn dist}) are equivalent to 
\begin{equation} \label{first dist eqn}
\int_\mathbb{D} \bigg( a \frac{\partial \phi}{\partial y} - b\frac{\partial \phi}{\partial x} \bigg) dxdy = 0 ,
\end{equation}
\begin{equation} \label{second dist eqn}
\int_\mathbb{D} \bigg( a \frac{\partial \phi}{\partial x} + b\frac{\partial \phi}{\partial y} \bigg) dxdy = -\int_\mathbb{D}P \phi \, dxdy , 
\end{equation}
for any test functions $\phi \in C_0^\infty(\mathbb D)$. 

For any $\psi \in C^\infty_0(\mathbb D)$, set $\phi = \frac{\partial \psi}{\partial y}$ in equation (\ref{first dist eqn}), $\phi = \frac{\partial \psi}{\partial x}$ in equation (\ref{second dist eqn}) and cancel the cross term $b\frac{\partial ^2 \psi}{\partial x\partial y}$ by taking summation of the two, we have 
\[
\int_\mathbb{D} a \,\Delta \psi \,dxdy=  -\int_\mathbb{D} P\, \frac{\partial \psi}{\partial x} \,dxdy 
\]
where $\Delta$ is the Laplace operator in the euclidean metric on ${\mathbb D}$. 

Similarly, set $\phi = \frac{\partial \psi}{\partial x}$ in equation (\ref{first dist eqn}), $\phi = \frac{\partial \psi}{\partial y}$ in equation (\ref{second dist eqn}) and take the difference of the two equations, we obtain
\[
\int_\mathbb{D} b \,\Delta \psi \,dxdy=  -\int_\mathbb{D} P\, \frac{\partial \psi}{\partial y} \,dxdy .
\]

We conclude now that $a$ and $b$ satisfy
\begin{equation} \label{a, b dist eqn}
\begin{split}
\Delta \,a &= \frac{\partial P}{\partial x} , \\ 
\Delta \,b &= \frac{\partial P}{\partial y} 
\end{split}
\end{equation}
on $\mathbb D$ in the sense of distribution. Now we apply the elliptic regularity theory for distributional solutions. As $F$ is smooth and $\alpha\in L^\infty (\mathbb D)$, we have $P \in L^2(\mathbb D)$. Hence the right hand side of equation (\ref{a, b dist eqn}) is in $H^{loc}_{-1}(\mathbb D)$. By the local regularity theorem (\cite{F}, Theorem 6.30), we have $a, b\in H_1^{loc} (\mathbb D)$. This implies $P \in H_1^{loc} (\mathbb D)$. Using this, we see that the right hand side of equation (\ref{a, b dist eqn}) is in $H_0^{loc}(\mathbb D)$. By the same local regularity theorem again, these implies $a, b\in H_2^{loc }(\mathbb D)$. Thus we can iterate this argument and see that $a, b \in H_s^{loc}(\mathbb D)$ for all integers $s$. By the Sobolev embedding theorem we have $a, b \in C^\infty (\mathbb D)$. Hence $\alpha_H$ can be extended to a smooth one form $\tilde \alpha$ on $\mathbb D$ and $d\tilde \alpha=0$ is satisfied on $\mathbb D$. 
\end{proof}

Now we proceed to the proof of Theorem \ref{no S2 lag shrinker}. 

\begin{proof}
 Let $F :\mathbb S^2 \to \mathbb C^2$ be a branched conformal Lagrangian self-shrinker with branch points $b_1, \cdots, b_k$. By Proposition \ref{a extend across B}, there is a smooth 1-form $\tilde \alpha$ on $\mathbb S^2$ such that $\tilde \alpha = \alpha_H$ on $\mathbb S^2\setminus B$. As $\tilde \alpha$ is closed and the first cohomology group of $\mathbb S^2$ is trivial, there is a smooth function $f$ on $\mathbb S^2$ such that $df = \tilde \alpha$. By equation (\ref{d^*a lemma}), $f$ satisfies
\begin{equation} \label{eqn of f}
 \Delta_g f = \frac 12 \,df (\nabla_g |F|^2)
\end{equation}
on $\mathbb S^2\setminus B$. Note that this equation is elliptic but not uniformly elliptic on $\mathbb S^2\setminus B$. By the strong maximum principle, the maximum of $f$ cannot be attained in $\mathbb S^2\setminus B$ unless $f$ is constant. Let $b\in B$ be a point where $f$ attains its maximum. Let $\mathbb D$ be a local chart around $b$ such that $g_{ij} = \lambda \delta_{ij}$ on $\mathbb D ^*$. As $\Delta_g = \lambda^{-1}\Delta$ and $\nabla_g = \lambda^{-1} D$, equation (\ref{eqn of f}) can be written  
\begin{equation} \label{eqn of f eu}
\Delta f =  \frac 12 \,df(D|F|^2 )\ \ \ \  \text{on }\mathbb D^* .
\end{equation}
As $f$ is smooth on $\mathbb D$, equation (\ref{eqn of f eu}) is in fact satisfied on $\mathbb D$. By the strong maximum principle, $f$ is constant as $f$ has an interior maximum at $b$. Hence $\alpha_H = 0$ and $\vec H=0$. This implies that $F$ is a branched minimal immersion in $\mathbb C^2$, which is not possible as $\mathbb S^2$ is compact.
\end{proof}

\section{Proof of Theorem \ref{main theorem}}
Let $F_n$ be a sequence of Lagrangian self-shrinkers which satisfies the hypothesis in the theorem. We need the following two lemmas for proving convergence of  $F_n$. 

The first lemma asserts that the images of $F_n$, for all $n$,  stay in a bounded region whose size depends only on the area bound $\Lambda$. In the following we use $B_r$ to denote the open ball in $\mathbb C^2$ with radius $r$ centered at the origin.  

\begin{lem} \label{F in U} Let $F_n$ be a sequence of two dimensional self-shrinkers in ${\mathbb R}^4$ with a uniform area upper bound. 
Then the image of $F_n$ lies in $B_{R_0}$, where $R_0$ depends only on the area upper bound. 
\end{lem}

\begin{proof}
It is proved in \cite{CL} that  for any $m$ dimensional self-shrinker $F :\Sigma^m \to \mathbb R^{m+k}$.
\begin{equation} \label{Delta |F|^2}
\Delta_g |F|^2 =  - 2|F^\perp|^2 + 2m .
\end{equation} 
 When $\Sigma$ is compact, there is a point $p\in \Sigma$ where $|F|^2$ attains its minimum. Then we have 
$$
\nabla_g |F|^2 = 2 F^\top = 0\ \  \text{  and  }\ \ \Delta_g F \geq 0 \ \ \text{  at  }p.
$$ 
This implies 
$$
0 \leq -2|F|^2 +2m \ \ \text{at $p$}.
$$
Hence $|F|^2 \leq m$ at $p$. Thus each $F_n$ must intersect the ball with radius $\sqrt 2$. 

On the other hand, we integrate equation (\ref{F in U}) and use equation (\ref{self similar eqn}) to get 
(also derived in \cite{CL}) 
\begin{equation} \label{W = A}
\mathcal W (F_n) :=\frac{1}{4} \int_\Sigma |H_n|^2 d\mu_n = \frac{1}{2} \mu (F_n) ,
\end{equation}
where $\mu(F_n)$ is the area of the immersion $F_n : \Sigma \to \mathbb R^{m+k}$. 

Next we use Simon's diameter estimate (\cite{S}, see also (2.1) in \cite{KLi}), which says that there is a constant $C$ such that 
\begin{equation} \label{bound in diam}
  \bigg(\frac{\mu (F_n)}{ \mathcal W(F_n)}\bigg)^{\frac{1}{2}} \leq  \text{diam}F_n(\Sigma) \leq C\big(  \mu (F_n) \mathcal W(F_n)\big)^{\frac{1}{2}} .
\end{equation}
Here diam$F_n(\Sigma)$ is the extrinsic diameter given by 
\[
\text{diam}F_n(\Sigma) := \sup_{x, y\in \Sigma} |F_n(x) - F_n(y)|.
\]
As the areas are uniformly bounded, the images of $F_n$ all lie in $B_{R_0}$ for some $R_0$ depending only on the area upper bound. 
\end{proof}

Let $U = B_{R_0+1}$ endowed with the metric $G$ given by
\begin{equation} \label{conformal metric on C^2}
G = e^{-\frac{|x|^2}{2}} \il \cdot, \cdot\li  .
\end{equation} 

The next lemma enables us to apply the results in \cite{P}  for harmonic maps into a compact Riemannian manifold. 

\begin{lem} \label{U embed in N}
There is a compact Riemannian manifold $(N, \overline g)$ such that $(U, G)$ isometrically embedded into $(N, \overline g)$. 
\end{lem}

\begin{proof}
Let $d = \frac{1}{R_0+1}$ and $N$ is the disjoint union of $B_{R_0 + 2}$ and $B_d$, with the identification that $x\sim y$ if and only if $y = \frac{x}{|x|^2}$ by the inversion. The manifold $N$ is compact, as it can be identified as the one point compactification of $\mathbb C^2$ via the stereographic projection.  Let $g_1$ be any metric on $B_d$. Let $\rho_1, \rho_2 \in C^\infty(N)$ be a partition of unity subordinate to the open cover $\{ B_{R_0+2}, B_d \}$ in $N$ and define a Riemannian metric on $N$ by $$\overline g = \rho_1 G + \rho_2 g_1.$$ As 
$$B_{R_0+2} \cap B_d = \{x \in B_{R_0+2}:  R_0+1 < |x| < R_0+2\},$$
$\bar g = \rho_1 G + \rho_2 g_1 = G$ on $B_{R_0+1}\subset B_{R_0+2}$. Thus the inclusion $U\subset B_{R_0+1} \subset N$ is an isometric embedding of $U$. 
\end{proof}

Now we are ready to prove Theorem \ref{main theorem}.

\begin{proof} As $F_n(\Sigma) \subset U$ for all $n$, we also treat $F_n$ as a map with image in $N$. As $F_n$ is a conformal minimal immersion with respect to the metric $h_n$ on $\Sigma$ and $G = \bar g$ on $U$, $F_n :(\Sigma , h_n) \to (N, \bar g)$ is a sequence of harmonic maps. The area of $F_n(\Sigma)$ in $(N,\bar{g})$ is 
$$
\tilde{\mu}(F_n) = \int_\Sigma e^{-\frac{|F_n|^2}{2}}d\mu_{F_n^*\langle\cdot,\cdot\rangle}.
$$
Therefore, as $e^{-\frac{|F_n|^2}{2}} \leq 1$ we have 
$$
\tilde \mu (F_n) \leq \mu (F_n)<\Lambda. 
$$
As $F_n:(\Sigma,h_n)\to (N,\bar g)$ is conformal, the area $\tilde \mu (F_n)$ in $N$ is the same as the Dirichlet energy:
\begin{equation*}
E_{h_n,\bar g}(F_n)=\tilde \mu (F_n) <\Lambda.
\end{equation*}
By assumption, the sequence of metrics $h_n$ converges to $h$. Hence we can apply the theory of bubble tree convergence of harmonic maps developed in \cite{P}. In particular, by passing to subsequence if necessary, there is a finite (possibly empty) set $S\subset \Sigma$  and a harmonic map $F_\infty :(\Sigma, h) \to (N, \bar g)$ such that $F_n \to F_\infty$ locally in $C^\infty(\Sigma \setminus S)$ \footnote{In \cite{SU1}, it is shown that the convergence is in locally $C^1(\Sigma\setminus S)$. As each $F_n : (\Sigma , h_n) \to (\mathbb C^2, G)$ is harmonic and $h_n \to h$, by the standard elliptic estimates (Chapter 6 of \cite{GT}) and a bootstrapping argument, there are constants $C(m)$ such that 
$$
||F_n||_{C^m} \leq C(m)
$$
for all $n \in \mathbb N$ . Using Arzela-Ascoli Theorem and picking a diagonal subsequence, one shows that a subsequence of $\{F_n\}$ converges locally smoothly to a smooth mapping $\Sigma \to \mathbb C^2$, which is $F_\infty$. }. For each $x \in S$, there is a finite sequence of nontrivial harmonic maps $s_{x, i} :\mathbb S^2 \to N$, $i=1, \cdots, n_x$ such that the renormalized maps $F_{n, I}$ (See \cite{P} for the details) of $F_n$ converges to the harmonic bubble tree map formed by $F_\infty$ and $s_{x,i}$. In particular, we have the Hausdorff convergence 
$$
F_n(\Sigma) \to F_\infty(\Sigma)\cup \bigcup_{x\in S}  \bigcup_{i} s_{x, i} (\mathbb S^2).
$$
As a result, $F_\infty$ and $s_{x,i}$ all have image in $\overline U \subset N$. 

The harmonic maps $s_{x,i}: \mathbb S^2 \to \overline U$ are branched conformal minimal immersions. From the construction of bubble tree convergence \cite{P},  a bubble at $x\in \Sigma$ is constructed by considering the renormalization 
$$
\tilde f_n(z) := f_n(\exp_x( \lambda_n z + c_n)),
$$ 
where $c_n, \lambda_n \to 0$ as $n\to\infty$. As the normalization is only on the domain, $\tilde f_n$ are also Lagrangian. The regularity results in \cite{SU} show that there is a finite set $\tilde S \subset \mathbb C$ such that $\tilde f_n$ converges locally in $C^\infty(\mathbb C \setminus \tilde S)$ to a harmonic map $\tilde f_\infty : \mathbb S^2 \to N$. Thus the bubble $\tilde f_\infty$ is Lagrangian. This applies to each level of the bubble tree, so all bubbles formed are also Lagrangian. Using Lemma \ref{U embed in N} and Lemma \ref{self similar iff minimal wrt g}, each $s_{x, i}$ is a branched Lagrangian self-shrinker in $\mathbb C^2$. By Theorem \ref{no S2 lag shrinker}, such an immersion does not exist. Thus there is no bubble formed in the convergence process (that is, $S$ is empty), and the convergence $F_n \to F_\infty$ is in $C^\infty(\Sigma)$. 

By Theorem 3.3 in \cite{SU}, as each $F_n : \Sigma \to N$ is nontrivial by definition, there is an $\epsilon >0$ such that $E_{h_n,\bar g}(F_n) \geq \epsilon$ for all $n\in \mathbb N$. Using the energy identity \cite{J}, we have 
\begin{equation}
 E_{h_n,\bar g}(F_n) \to E_{h,\bar g}(F_\infty) \ \ \ \text{as }n\to \infty .
\end{equation}
Hence $E_{h,\bar g}(F_\infty) \geq \epsilon$ and $F_\infty$ is nontrivial (One can also use the estimate of diameter and the Hausdorff convergence as in the proof of Theorem \ref{main thm 2} in the next section to show that $F_\infty$ is nontrivial).

As the convergence $F_n \to F_\infty$ is in $C^\infty(\Sigma)$ and the metrics $h_n$ converges to $ h$ smoothly, the harmonic map $F_\infty : (\Sigma, h) \to (N, \bar g)$ is also conformal. Thus, $F_\infty$ is a branched conformal minimal immersion from $(\Sigma,h)$ into $(N,\bar g)$. It is also a Lagrangian immersion. Since the image of $F_\infty$ lies in $\overline U$, we also view $F_\infty$ as a branched Lagrangian minimal immersion in $\mathbb C^2$ with respect to the metric $G$ defined by (\ref{conformal metric on C^2}). By Lemma \ref{self similar iff minimal wrt g}, $F_\infty$ is then a Lagrangian branched conformal self-shrinker. 
\end{proof}

\section{Proof of Theorem \ref{main thm 2}}
\begin{proof} Let $\{F_n\}$ be a sequence of immersed Lagrangian self-shrinking tori in $\mathbb C^2$ with a uniform area upper bound. Let $h_n$ be the metric on the torus $T$ which is conformal to $F_n^*\il \cdot, \cdot\li$ and with zero Gauss curvature. If we can show that $h_n$ stays in a bounded domain of the Teichm\"uller space, then (i) follows from Theorem \ref{main theorem}. 

Using Lemma \ref{F in U}, Lemma \ref{U embed in N} and Lemma \ref{self similar iff minimal wrt g}, we also treat each $F_n$ as a minimal immersion in $N$, this means that $F_n:(T, h_n)\to (N,\bar{g})$ is conformal and harmonic. Assume the contrary that the conformal structures degenerate. In this case, there is a mapping $\hat{F}_\infty$ from $\Sigma_\infty$ to $N$  and the  image $F_n(T)$ converges in the Hausdorff distance to $\hat F_\infty(\Sigma_\infty)$ in $N$. Here $\Sigma_\infty$ is a stratified surfaces $\Sigma_\infty = \Sigma_0 \cup \Sigma_b$ formed by the {\it principal component} $\Sigma_0$ and {\it bubble component} $\Sigma_b$. The principal component $\Sigma_0$ is formed by pinching several closed, homotopically nontrivial curves in $T$ and the bubbling component is a union of spheres. There are no necks between the components since $F_n$ is conformal. The map $\hat F_\infty$ is continuous on $\Sigma_\infty$  and harmonic when restricted to each component of $\Sigma_\infty$. Since all the components intersect each others possibly at finitely many points, $\hat F_\infty$ is harmonic except at a finite set  $\hat S$. 

Since the conformal structures determined by the metrics $h_n$ degenerate, at least one homotopically nontrivial closed curve must be pinched to a point as $n\to\infty$. It follows that $\Sigma_0$ is a finite union of $\mathbb S^2$'s. Each of these 2-spheres is obtained by adding finitely many points to the cylinder ${\mathbb S}^1\times{\mathbb R}$ that comes from pinching one or two homotopically nontrivial loops: two at the infinity and at most finitely many at the blowup points of the sequence $F_n$, by the removable singularity theorem of Sacks-Uhlenbeck \cite{SU}. Therefore, $\hat F_\infty$ is a finite union of harmonic mappings $\hat F_\infty^i$ from the sphere to $N$. 

Since $F_n$ are conformal, there are no necks between the components. The bubble tree convergence described above are given by the results in \cite{CT} or \cite{CLi}. In \cite{CT}, the limiting surface is a stratified surface with geodesics connecting the two dimensional components. But together with conformality of each $F_n$ and Proposition 2.6 in \cite{CT}, one sees that all the geodesics involved have zero length. Alternatively, we can use the compactness theorem in \cite{CLi} which says: Suppose that $\{f_k\}$ is a sequence of $W^{2,2}$  branched conformal immersions of $(\Sigma, h_k)$ in a compact manifold $M$. If
$$
\sup_k\left\{\mu(f_k)+W(f_k)\right\}<+\infty
$$
then  either $\{f_k\}$ converges to a point, or
there is a stratified sphere $\Sigma_\infty$ and a $W^{2,2}$ branched conformal immersion
$f_\infty:\Sigma_\infty\to M$,
such that a subsequence of
$\{f_k(\Sigma)\}$ converges to  $f_\infty(\Sigma_\infty)$
in the Hausdorff topology, and the area and the Willmore energy satisfy
$$
\mu(f_\infty)=\lim_{k\rightarrow+\infty}\mu(f_k)\,\, \mbox{and} \,\, W(f_0)\leq\lim_{k\rightarrow+\infty} W(f_k).
$$
 The conditions of the theorem are satisfied by the sequence $\{F_n\}$ as the area $\tilde \mu (F_n)$ and the Willmore energy $W(F_n)$ in $N$ (which is zero as each $F_n$ is minimal immersion in $N$) are uniformly bounded. In our situation, $F_k$ will not converges to a point since $\text{diam}(F_n(T)) \geq \sqrt 2$ by equations (\ref{bound in diam}) and (\ref{W = A}). Thus the inequality on the Willmore energies in $N$
$$
W(\hat F_\infty)\leq\lim_{n\to\infty}W(F_n) = 0
$$
implies that the limiting $\mathbb S^2$'s are all branched minimal surfaces in $N$. 

Consequently, the images $F_n(T)$ converge in the Hausdorff distance to the image of finitely many harmonic maps $\mathbb S^2 \to N$. These harmonic maps are conformal branched immersions, which are also Lagrangian by similar reasons as in the proof of Theorem \ref{main theorem}. By Theorem \ref{no S2 lag shrinker}, all these harmonic maps are trivial. Hence, the images $F_n(T)$ converge in the Hausdorff distance to a point in $\mathbb C^2$. Again, this is impossible by the diameter estimate $\text{diam}(F_n(T)) \geq \sqrt 2$. This contradiction shows that the conformal structures cannot degenerate and that finishes the proof of (i).

To show (ii), we assume $F_n\to F_\infty$ smoothly by (i). Note that when the area upper bound $\Lambda<16\pi$, by equation (\ref{W = A}), the Willmore energy ${\mathcal W}(F_n)$ of $F_n$ in $({\mathbb R}^4, \langle \cdot,\cdot\rangle)$ admits an upper bound $\Lambda/2<8\pi$. By Theorem 1 in \cite{CLi},
$$
{\mathcal W}(F_\infty)\leq \lim_{n\to\infty}{\mathcal W}(F_n),
$$ 
thus the limiting Lagrangian shrinker $F_\infty$ also has Willmore energy strictly less than $8\pi$. By Proposition 4.1 in \cite{KLi} (also, Theorem 3.1 there), the limiting mapping $F_\infty: \Sigma \to \mathbb C^2$ has no branch points, therefore it is an immersion. By Theorem 6 in \cite{LY}, the immersion $F_\infty$ must be an embedding. 
\end{proof}

\bibliographystyle{amsplain}

\end{document}